\documentclass[11pt]{article}
\usepackage{amssymb,epsfig,amsmath,natbib,pifont,graphicx}
\graphicspath{{./}{../graphs/}}
\newcommand{\ud}{\mathrm{d}}
\newcommand{\ign}{\mathrm{LS}}
\newenvironment{proof}[1][Proof]{\noindent\textbf{#1.} }{\ \rule{0.5em}{0.5em}}
\title{Contrasting Probabilistic Scoring Rules}
\author{\Large{Reason L. Machete}\footnote{email: {\tt r.l.machete@lse.ac.uk}, tel: +44(0)118 378 6378}\\\\{\normalsize Dept. of Mathematics and Statistics, P. O. Box 220, Reading, RG6 6AX, UK}}
\date{Dated: May 24, 2012}
\newtheorem{proposition}{Proposition}[section]

\begin{document}
\maketitle
\begin{abstract}
There are several scoring rules that one can choose from in order to score probabilistic forecasting models or estimate model parameters. Whilst it is generally agreed that proper scoring rules are preferable, there is no clear criterion for preferring one proper scoring rule above another. This manuscript compares and contrasts some commonly used proper scoring rules and provides guidance on scoring rule selection. In particular, it is shown that the logarithmic scoring rule prefers erring with more uncertainty, the spherical scoring rule prefers erring with lower uncertainty, whereas the other scoring rules are indifferent to either option. 
\end{abstract}
\section{Introduction}
Issuing probabilistic forecasts is meant to express uncertainty about the future evolution of some quantity of interest. Such forecasts arise in many applications such as macroeconomics, finance, weather and climate forecasting. There are several scoring rules that one can choose from in order to elicit probabilistic forecasts, rank competing forecasting models or estimate forecast distribution parameters. It is generally agreed that one should select scoring rules that encourage a forecaster to state his `best' judgement of the distribution, the so called {\em proper} scoring rules~\citep{frie-83,nau-85,til-07}, but which one to use is generally an open question. We shall take scoring rules to be loss functions that a forecaster wishes to minimise. Scoring rules that are minimised if and only if the issued forecasts coincide with the forecaster's best judgement are said to be {\em strictly proper}~\citep{til-07,joc-06}. We shall restrict our attention to strictly proper scoring rules. 

Nonetheless using scoring rules to rank competing forecasting models poses a problem; scoring rules do not provide a universally acceptable ranking of performance. In estimation, different scoring rules will yield different parameter estimates~\citep{til-07,joh-11}. Moreover, a forecaster's best judgement may depart from the ideal; the ideal is a distribution that nature or the data generating process would give~\citep{til-gne}. Although strictly proper scoring rules encourage experts to issue their best judgements, such judgements may yet differ from each other and the ideal. Which scoring rule should one use to choose between two experts? \cite{sav-71} made the instructive statement that ``any criteria for distinguishing among scoring rules must arise out of departures of actual subjects from the ideal.'' There have been some efforts to contrast scoring rules, but none seem to have followed this insight. 

\cite{bick-07} made empirical comparisons of the quadratic, spherical and logarithmic scoring rules and found them to yield different rankings of competing forecasts but failed to see why. Considering a concave nonlinear utility function that explicitly depends on the scoring rule, he also found the logarithmic scoring rule to yield the least departures from  honest opinions at maximal utility, a point he claimed favours it as a rule of choice. But a utility function need not be exponential nor explicitly depend on the scoring rule.~\cite{jos-08} considered weighted scoring rules and showed that they correspond to different utility functions. A limiting feature of the utility functions considered is that they are defined on bounded intervals; there are many applications in which the variable of interest is unbounded. Their motivation for weighted scoring rules is based on betting arguments, but it is not clear what the betting strategies (if any) are. Recently,~\cite{boe-11} empirically compared the Quadratic Probability Score (QPS), Ranked Probability Score (RPS) and the logarithmic scoring rule on UK inflation forecasts by the Monetary Policy Committee and the Survey of External Forecasters (SEF). They found the scoring rules to rank the two sets of distributions similarly. Upon ranking individual forecasters from the SEF, they found the RPS to have better discriminatory power than the QPS, a feature they attributed to the RPS's sensitivity to distance. Despite the foregoing efforts, there is lacking a theoretical assessment of what the preferences of the commonly used scoring rules are with respect to the ideal.

This paper contrasts how different scoring rules would rank competing forecasts of specified departures from ideal forecasts and provides guidance on scoring rule selection. It focuses upon those scoring rules that are commonly used in the forecasting literature, including econometrics and meteorology. More specifically, we contrast the relative information content of forecasts preferred by different scoring rules. Implications of the results on decision making are then suggested, noting that it may be desirable to be more or less uncertain when communicating probabilistic forecasts. We realise that an appropriate utility function may be unknown~\citep{bick-07} and expected utility theory may not even be appropriate~\citep{kah-79}. 

In section~\ref{sec:cat}, we consider the case of scoring categorical forecasts by the Brier score~\citep{brie}, the logarithmic scoring rule~\citep{good-52} and the spherical scoring rule~\citep{frie-83}. For simplicity, special attention is focused on binary forecasts. This section then inspires our study of density forecasts in section~\ref{sec:dens}, where we consider three scoring rules: the Quadratic Score~\citep{til-07}, Logarithmic Score~\citep{good-52}, Spherical Score~\citep{frie-83} and Continuous Ranked Probability Score~\citep{eps-69}. We conclude with a discussion of the results in section~\ref{sec:disc}.
\section{Categorical Forecasts}
\label{sec:cat}
In this section, we consider the scoring of categorical forecasts. The scoring rules considered are Brier score~\citep{brie}, the logarithmic scoring rule and the spherical scoring rule~\citep{frie-83}. In order to aid intuition in the next section, here we focus on the binary case. Another commonly used scoring rule for categorical forecasts is the Ranked Probability Score (RPS)~\citep{eps-69}. In the binary case, the RPS score reduces to the Brier score.

It will be useful to be aware of the following basics. Given any vectors $\boldsymbol{f},\boldsymbol{g}\in\Re^m$, the {\em inner product} between the two vectors is $$\langle\boldsymbol{f},\boldsymbol{g}\rangle=\sum_{i=1}^mf_ig_i,$$
from which the $L_2$-norm is defined by $||\boldsymbol{f}||_2=\langle\boldsymbol{f},\boldsymbol{f}\rangle^{1/2}$.
\subsection{The Brier score}
Consider a probabilistic forecast $\{f_i\}_{i=1}^m$ of $m$ categorical events. Suppose the true distribution is $\{p_i\}_{i=1}^m$. If the actual outcome is the $j$th category, the Brier score is given by~\citep{brie}
\begin{equation*}
BS(\boldsymbol{f},j)=\frac{1}{m}\sum_{i=1}^m(f_i-\delta_{ij})^2,
\end{equation*}
where $\delta_{ij}=0$ if $i\neq j$ and $\delta_{ij}=1$ if $i=j$. If follows that if we expand out the bracket we get
\begin{equation*}
BS(\boldsymbol{f},j)=\frac{1}{m}\left(\sum_{i=1}^mf_i^2-2f_j+1\right).
\end{equation*}
The expected Brier score is then given by
\begin{align*}
\mathbb{E}[BS(\boldsymbol{f},J)]&=\sum_{j=1}^mp_jBS(f,j)\displaybreak[0]\\
&=\frac{1}{m}\sum_{i=1}^m\left(f_i^2-2p_if_i+p_i\right)\displaybreak[0]\\
&=\frac{1}{m}\sum_{i=1}^m\left[\left(f_i-p_i\right)^2+p_i-p_i^2\right]\displaybreak[0]\\
&=\frac{1}{m}\left\{||\boldsymbol{\gamma}||_2^2+\sum_{i=1}^mp_i(1-p_i)\right\},
\end{align*}
where $\boldsymbol{\gamma}$ is a vector with components $\gamma_i=f_i-p_i$ for all $i=1,\ldots,m$. It is evident from the last expression on the right hand side that the Brier score is effective with respect to the metric $d_2(\boldsymbol{f},\boldsymbol{g})=||\boldsymbol{f}-\boldsymbol{g}||_2$. When $m=2$, we can put $f_1=p+\gamma$, $p_1=p$ and $p_2=q$ and obtain
\begin{equation*}
\mathbb{E}[BS(\boldsymbol{f},J)]=\gamma^2+pq.
\end{equation*}
It follows that $\pm\gamma$ will yield the same Brier score. This means the Brier score does not discriminate between over-estimating and under-estimating the probabilities with the same amount. Further more, for any two forecasts $\boldsymbol{f}_i=(p+\gamma_i,q-\gamma_i)$, $i=1,2$, with $|\gamma_1|<|\gamma_2|$, the Brier score would prefer the forecast corresponding to $\gamma_1$.
\subsection{Logarithmic scoring rule}
The logarithmic scoring rule was proposed by~\cite{good-52}. It was later termed Ignorance by~\cite{roul-02} when they introduced it to the meteorological community. Given a probabilistic forecast $\boldsymbol{f}=(f_1,f_2,\ldots,f_m)$, the logarithmic scoring rule is given by $\ign(\boldsymbol{f},j)=-\log f_j$, where $j$ denotes the category that materialises. Let us consider the expected logarithmic score of the forecasting scheme $\boldsymbol{f}=(p+\gamma,q-\gamma)$:
\begin{equation}
\mathbb{E}[\ign(\boldsymbol{f},J)]=-p\log(p+\gamma)-q\log(q-\gamma),
\label{eqn:ign1}
\end{equation}
where $J\in\{1,2\}$ is a random variable. The above expectation is also referred to as the Kullback-Leibler Information Criterion~\citep{gra-06}. As noted by~\cite{frie-83}, this scoring rule is not effective. 

If we let $\boldsymbol{f}_+=(p+\gamma,q-\gamma)$ and $\boldsymbol{f}_-=(p-\gamma,q+\gamma)$, then we can define $\mathbb{E}[\ign]_{\pm}=\mathbb{E}[\ign(\boldsymbol{f}_+,J)]-\mathbb{E}[\ign(\boldsymbol{f}_-,J)]$. Then, assuming that $\gamma>0$ without loss of generality, 
\begin{equation}
\mathbb{E}[\ign]_{\pm}=p\log\left(\frac{p-\gamma}{p+\gamma}\right)+q\log\left(\frac{q+\gamma}{q-\gamma}\right)
\label{eqn:ign2}
\end{equation}
Note that when $p=q=0.5$, then $\mathbb{E}[\ign]_{\pm}=0$, otherwise $\mathbb{E}[\ign]_{\pm}\neq0$. Differentiating~(\ref{eqn:ign2}) with respect to $\gamma$ yields
\begin{equation}
\frac{\ud}{\ud\gamma}\mathbb{E}[\ign]_{\pm}=\frac{2\gamma^2(p^2-q^2)}{(p^2-\gamma^2)(q^2-\gamma^2)}
\label{eqn:ign3}
\end{equation}
Expressions~(\ref{eqn:ign2}) and~(\ref{eqn:ign3}) are well defined provided $\gamma<\min(p,q)$.
\begin{eqnarray*}
\frac{\ud}{\ud\gamma}\mathbb{E}[\ign]_{\pm}>0,\quad\mbox{if $p>q$}\\
\frac{\ud}{\ud\gamma}\mathbb{E}[\ign]_{\pm}<0,\quad\mbox{if $p<q$}
\end{eqnarray*}
It follows that $\mathbb{E}[\ign]_{\pm}>0$ if $p>q$ and $\mathbb{E}[\ign]_{\pm}<0$ if $p<q$. In other words, the logarithmic score penalises over confidence on the likely outcome and rewards erring on the side of caution. Given forecasting schemes that are equally calibrated, the logarithmic score will prefer the one with a higher entropy. To explain this further, let us denote the entropy of the forecast corresponding to $\gamma$ by $h(\gamma)$, i.e.
\begin{equation}
h(\gamma)=-(p+\gamma)\log(p+\gamma)-(q-\gamma)\log(q-\gamma).
\label{eqn:ent1}
\end{equation}
We now define the function $G(\gamma)=h(\gamma)-h(-\gamma)$ and claim that $G(\gamma)<0$ for $0<\gamma<q<p$. To prove this claim, we first note that $G(0)=0.$ It then suffices to show that $G'(0)<0.$ Note that
\begin{eqnarray*}
G'(\gamma)&=&-\log(p+\gamma)+\log(q-\gamma)-\log(p-\gamma)+\log(q+\gamma)\\
 &=&-\log\left(\frac{p+\gamma}{q+\gamma}\right)+\log\left(\frac{q-\gamma}{p-\gamma}\right).
\end{eqnarray*}
The condition $p>q$ implies that $G'(\gamma)<0$ for all $\gamma\in(0,q)$. Therefore, it is evident that, of the two forecasts, the logarithmic score prefers the one with a higher entropy. We have thus proved the following proposition:
\begin{proposition}
\label{prop4}
Given two forecasts, $\boldsymbol{f}_+=(p+\gamma,q-\gamma)$ and $\boldsymbol{f}_-=(p-\gamma,q+\gamma)$, where $0<\gamma<q<p$, the logarithmic scoring rule prefers $\boldsymbol{f}_-$. Moreover, $\boldsymbol{f}_-$ has a higher entropy than $\boldsymbol{f}_+$.
\end{proposition}

What about when there are two forecasts $\boldsymbol{f}_i=(p+\gamma_i,q-\gamma_i)$, $i=1,2$ with $0<\gamma_1<\gamma_2<q$ and $p>q$? It is obvious that the Brier score will prefer $\boldsymbol{f}_1$ over $\boldsymbol{f}_2$. The question is, which of the two forecasts will the logarithmic scoring rule prefer? We answer this question by stating the following proposition: 
\begin{proposition}
\label{prop2}
Given two forecasts $\boldsymbol{f}_i=(p+\gamma_i,q-\gamma_i)$, $i=1,2$ with $0<\gamma_1<\gamma_2<q$ and $p>q$, the logarithmic scoring rule prefers $\boldsymbol{f}_1$ over $\boldsymbol{f}_2$.
\end{proposition}
\begin{proof} In order to prove this proposition, it is sufficient to consider the expected logarithmic score of the forecast $\boldsymbol{f}=(p+\gamma,q-\gamma)$, which is given by equation~(\ref{eqn:ign1}). Differentiating the equation with respect to $\gamma$ yields
\begin{equation}
\frac{\ud}{\ud\gamma}\mathbb{E}[\ign(\boldsymbol{f},J)]=\frac{\gamma}{(p+\gamma)(q-\gamma)}
\label{eqn:ign5}
\end{equation}
Equation~(\ref{eqn:ign5}) implies that, if $q>\gamma>0$, $\mathbb{E}[\ign(\boldsymbol{f},J)]$ is an increasing function of $\gamma$. Hence, the logarithmic scoring rule prefers the forecast $\boldsymbol{f}_1$.

On the other hand, if $\gamma<0$ with $|\gamma|<p$, then equation~(\ref{eqn:ign5}) implies that $\mathbb{E}[\ign(\boldsymbol{f},J)]$ is a decreasing function of $\gamma$. It then follow that, given  $\gamma_2<\gamma_1<0$ with $|\gamma_2|<p$, the logarithmic scoring rule will prefer the forecast $\boldsymbol{f}_1$.
\end{proof}

Finally, let us consider the case of two forecasts $\boldsymbol{f}_1=(p+\gamma_1,q-\gamma_1)$ and $\boldsymbol{f}_2=(p-\gamma_2,q+\gamma_2)$, where $0<\gamma_1<\gamma_2<q<p$. Again, it is clear that the Brier score will prefer the forecast $\boldsymbol{f}_1$ over $\boldsymbol{f}_2$. It remains to be seen which forecast the logarithmic scoring rule will prefer. This may be determined by considering the function $H(\gamma_1,\gamma_2)$, where 
\begin{equation}
H(\gamma_1,\gamma_2)=p\log\left(\frac{p-\gamma_2}{p+\gamma_1}\right)+q\log\left(\frac{q+\gamma_2}{q-\gamma_1}\right)
\label{eqn:ign6}
\end{equation}
Note that $H(\gamma_1,\gamma_2)=\mathbb{E}[\ign(\boldsymbol{f}_1,J)]-\mathbb{E}[\ign(\boldsymbol{f}_2,J)]$. The forecast $\boldsymbol{f}_1$ is preferred if $H(\gamma_1,\gamma_2)<0.$ The following proposition gives insights of relative forecast performance in the parameter space. 
\begin{proposition}
\label{prop1}
Given that $0<\gamma_2<q<p$, there exists $\gamma^*\in(0,\gamma_2)$ such that (a) $H(\gamma*,\gamma_2)=0$, (b) $H(\gamma_1,\gamma_2)>0$ for $\gamma_1\in(\gamma^*,\gamma_2)$ and (c) $H(\gamma_1,\gamma_2)<0$ for $\gamma_1\in(0,\gamma^*)$.
\end{proposition}
Before proving the above proposition, we remark that $H(\gamma_1,\gamma_2)<0$ if and only if the logarithmic scoring rule prefers the forecast $\boldsymbol{f}_1$. This proposition implies that the logarithmic scoring rule and the Brier score prefer different forecasts when $\gamma_1\in(\gamma^*,\gamma_2)$. Let us now consider the proof of this proposition. \\\\
\begin{proof}
In proving this proposition, it is useful to bear in mind that $H(\gamma_2,\gamma_2)>0$. The partial derivatives of equation~(\ref{eqn:ign6}) are given by 
\begin{equation}
\frac{\partial H}{\partial\gamma_1}=\frac{\gamma_1}{(p+\gamma_1)(q-\gamma_1)}\quad\mbox{and}\quad\frac{\partial H}{\partial\gamma_2}=\frac{-\gamma_2}{(p+\gamma_2)(q-\gamma_2)}.
\label{eqn:ign13}
\end{equation}
Further more, we can differentiate equations~(\ref{eqn:ign13}) to obtain
\begin{equation}
\frac{\partial^2H}{\partial\gamma_1^2}=\frac{pq+\gamma_1^2}{(p+\gamma_1)^2(q-\gamma_1)^2}\quad\mbox{and}\quad\frac{\partial^2H}{\partial\gamma_2^2}=\frac{-(pq+\gamma_2^2)}{(p-\gamma_2)^2(q+\gamma_2)^2}.
\label{eqn:ign14}
\end{equation}
It follows from equations~(\ref{eqn:ign13}) that $\partial H/\partial\gamma_1=0$ at $\gamma_1=0$ and $\partial H/\partial\gamma_2=0$ at $\gamma_2=0$. Since $\partial^2 H/\partial\gamma_1^2>0$ for all $\gamma_1$, $H(\gamma_1,\cdot)$ has a global minimum at $\gamma_1=0$. Similarly, $H(\cdot,\gamma_2)$ has a global maximum at $\gamma_2=0$ since $\partial^2 H/\partial\gamma_2^2<0$ for all $\gamma_2$ and the first partial derivative with respect to $\gamma_2$ vanishes there. In particular, $H(0,\gamma_2)\le H(0,0)=0$, i.e. $H(0,\gamma_2)\le0$. For $\gamma_2>0$, we have the strict inequality, $H(0,\gamma_2)<0$. But we also have $H(\gamma_2,\gamma_2)>0$ from Proposition~\ref{prop4}. It, therefore, follows from the intermediate value theorem that $H(\gamma_1,\gamma_2)=0$ for some $\gamma_1=\gamma^*\in(0,\gamma_2)$, which completes the proof.
\end{proof}

\begin{proposition}
\label{prop3}
For positive $\gamma_1$ and $\gamma_2$ such that $\gamma_1<q<p$ and $\gamma_2<p$, the entropy of the forecast $\boldsymbol{f}_1=(p+\gamma_1,q-\gamma_1)$ is lower than that of the forecast $\boldsymbol{f}_2=(p-\gamma_2,q+\gamma_2)$ whenever $\gamma_2\le(p-q)/2$.
\end{proposition}
A consequence of this proposition is that the forecast corresponding to $\gamma_1=\gamma^*$ is more informative than $\boldsymbol{f}_2$ provided $\gamma_2\le(p-q)/2$. Otherwise, either forecast could be more informative than the other. We now give the proof of this proposition.\\\\
\begin{proof}
To prove the above proposition, we consider the derivative of equation~(\ref{eqn:ent1}):
\begin{equation*}
\frac{\ud h}{\ud\gamma}=-\log\left(\frac{p+\gamma}{q-\gamma}\right).
\end{equation*}
We then note that $\ud h/\ud\gamma<0$ provided that $(p-q)>-2\gamma$. If $\gamma>0$, this inequality is trivially satisfied. On the other hand, if $\gamma<0$, then the inequality is satisfied provided $|\gamma|<(p-q)/2$. If $\gamma_2<(p-q)/2$, then $h(\gamma)$ is a strictly decreasing function for all $\gamma\in[-\gamma_2,\gamma_2]$, which implies that $h(\gamma_1)>h(\gamma_2)$. If $\gamma_2>(p-q)/2$, then $h(\gamma)$ is an increasing function for all $\gamma\in(-\gamma_2,-(p-q)/2)$ (provided $p>3q$) and strictly decreasing function in $(-(p-q)/2,\gamma_1)$, which implies that $h(-(p-q)/2)>\max\{h(\gamma_1),h(-\gamma_2)\}$. Hence, in this case, we cannot determine which of $h(\gamma_1)$ and $h(-\gamma_2)$ is lower.
\end{proof}
\subsection{The Spherical Scoring Rule}
The {\em spherical scoring rule} is given by
\begin{equation*}
S(\boldsymbol{f},j)=-\frac{f_j}{||\boldsymbol{f}||_2}.
\end{equation*}
Define $\boldsymbol{f}_-=\boldsymbol{p}-\boldsymbol{\gamma}$ and $\boldsymbol{f}_+=\boldsymbol{p}+\boldsymbol{\gamma}$. Which of the two forecasts $\boldsymbol{f}_-$ and $\boldsymbol{f_+}$ does the spherical scoring rule prefer? In order to address this question, we appeal to geometry. Considering $\boldsymbol{f}=\boldsymbol{p}+\boldsymbol{\gamma}$, the dot product rule yields
\begin{eqnarray*}
||\boldsymbol{p}||_2||\boldsymbol{f}||_2\cos\theta=\langle\boldsymbol{f},\boldsymbol{p}\rangle
\end{eqnarray*}
where $\theta$ is the angle between $\boldsymbol{f}$ and $\boldsymbol{p}$. The above formula may be rewritten as 
\begin{equation}
\cos\theta=\frac{||\boldsymbol{p}||_2^2+\langle\boldsymbol{\gamma},\boldsymbol{p}\rangle}{||\boldsymbol{p}||_2||\boldsymbol{f}||_2}.
\label{eqn:cos}
\end{equation}
We then state the following proposition:
\begin{proposition} 
\label{prop:cos}
If $\boldsymbol{p}=(p,q)$ and $\boldsymbol{\gamma}=(\gamma,-\gamma)$ and if we denote the right hand side of equation~(\ref{eqn:cos}) by $C(\gamma)$, then
\begin{equation*}
\frac{\ud C(\gamma)}{\ud\gamma}=\frac{-\gamma}{||\boldsymbol{p}||_2||\boldsymbol{f}||_2^3}.
\end{equation*}
\end{proposition}
\begin{proof}
First note that $||\boldsymbol{f}||_2^2=||\boldsymbol{p}||_2^2+2\langle\boldsymbol{\gamma},\boldsymbol{p}\rangle+||\boldsymbol{\gamma}||_2^2$ and 
\begin{equation*}
\frac{\ud||\boldsymbol{f}||_2}{\ud\gamma}=\frac{(p-q)+2\gamma}{||\boldsymbol{f}||_2}.
\end{equation*}
Using the quotient rule, we then differentiate $C(\gamma)$ with respect to $\gamma$ to obtain
\begin{align*}
\frac{\ud C(\gamma)}{\ud\gamma}&=\frac{||\boldsymbol{f}||_2\frac{\ud}{\ud\gamma}(||\boldsymbol{p}||_2^2+\langle\boldsymbol{\gamma},\boldsymbol{p}\rangle)-(||\boldsymbol{p}||_2^2+\langle\boldsymbol{\gamma},\boldsymbol{p}\rangle)\frac{\ud}{\ud\gamma}||\boldsymbol{f}||_2}{||\boldsymbol{p}||_2||\boldsymbol{f}||_2^2}\displaybreak[0]\\
&=\frac{||\boldsymbol{f}||_2^2(p-q)-(||\boldsymbol{p}||_2^2+\langle\boldsymbol{\gamma},\boldsymbol{p}\rangle)[(p-q)+2\gamma]}{||\boldsymbol{p}||_2||\boldsymbol{f}||_2^3}\displaybreak[0]\\
&=\frac{(p-q)(||\boldsymbol{p}||_2^+2\langle\boldsymbol{\gamma},\boldsymbol{p}\rangle+||\boldsymbol{\gamma}||_2^2)-(||\boldsymbol{p}||_2^2+\langle\boldsymbol{\gamma},\boldsymbol{p}\rangle)[(p-q)+2\gamma]}{||\boldsymbol{p}||_2||\boldsymbol{f}||_2^3}\displaybreak[0]\\
&=\frac{-2\gamma||\boldsymbol{p}||_2^2+\gamma(p-q)^2}{||\boldsymbol{p}||_2||\boldsymbol{f}||_2^3}\displaybreak[0]\\
&=\frac{-2\gamma||\boldsymbol{p}||_2^2+\gamma(||\boldsymbol{p}||_2^2-2pq)}{||\boldsymbol{p}||_2||\boldsymbol{f}||_2^3}\displaybreak[0]\\
&=\frac{-\gamma(||\boldsymbol{p}||_2^2+2pq)}{||\boldsymbol{p}||_2||\boldsymbol{f}||_2^3}\displaybreak[0]\\
&=\frac{-\gamma(p+q)^2}{||\boldsymbol{p}||_2||\boldsymbol{f}||_2^3}.
\end{align*}
The desired result follows from noting that $p+q=1$.
\end{proof}
\begin{proposition}
Suppose that $p>q$ and $\gamma\in(0,q)$. Then the spherical scoring rule prefers the lower entropy forecast, $\boldsymbol{f}_+$, instead of $\boldsymbol{f}_-$. 
\end{proposition}
\begin{proof}
Since the spherical scoring rule is effective, it suffices for us to show that $d_*(\boldsymbol{f}_+,\boldsymbol{p})<d_*(\boldsymbol{f},\boldsymbol{p})$. Suppose the angles that each of $\boldsymbol{f}_+$ and $\boldsymbol{f}_-$ makes with $\boldsymbol{p}$ are respectively $\theta_+$ and $\theta_-$. It is then true that $d_*(\boldsymbol{f}_+,\boldsymbol{p})<d_*(\boldsymbol{f},\boldsymbol{p})$ if and only $\theta_+<\theta_-$ since each distance is the length of a chord on a unit circle. Note that $C(0)=1$ and $C'(0)=0$. From Proposition~\ref{prop:cos}, $-C'(\tau)<C'(-\tau)$ for all $\tau\in(0,\gamma)$, which implies that 
\begin{eqnarray*}
-\int_0^{\gamma}C'(\tau)\ud\tau<\int_0^{\gamma}C'(-\tau)\ud\tau&\Rightarrow&-\int_0^{\gamma}C'(\tau)\ud\tau<-\int_0^{-\gamma}C'(\tau)\ud\tau\\
&\Rightarrow&\int_0^{\gamma}C'(\tau)\ud\tau>\int_0^{-\gamma}C'(\tau)\ud\tau\\
&\Rightarrow&\left. C(\tau)\right|_0^{\gamma}>\left. C(\tau)\right|_0^{-\gamma}\\
&\Rightarrow&C(\gamma)-C(0)>C(-\gamma)-C(0)\\
&\Rightarrow&C(\gamma)>C(-\gamma).
\end{eqnarray*}
But  $C(\gamma)>C(-\gamma)$  implies that $\theta_+<\theta_-$.
\end{proof}
\section{Density Forecasts}
\label{sec:dens}
This section considers scoring rules for for forecasts of continuous variables. It is in some sense a generalisation of the previous section. As before, we consider how each scoring rule would rank two competing predictive distributions of fairly good quality. In the case of the logarithmic scoring rule and the Continuous Ranked Probability Score, we consider errors of each predictive distribution, $f(x)$, from the target distribution, $p(x)$, that are odd functions, i.e. $\gamma(x)=f(x)-p(x)$ with $\gamma(-x)=-\gamma(x)$.

Familiarity with the following notation and definitions will be useful. Given two functions $f(x)$ and $g(x)$ that are bounded, an {\em inner product} is defined by $$\langle f,g\rangle=\int_{-\infty}^{\infty}f(x)g(x)\ud x.$$
Then the $L_2$ norm is defined to be $||f||_2=\langle f,f\rangle^{1/2}$. 
\subsection{The Quadratic scoring rule}
A continuous counterpart of the Brier score is the quadratic score~\citep{til-07}, given by
\begin{equation*}
QS(f,X)=||f||_2^2-2f(X),
\end{equation*}
where $X$ is a random variable. Taking the expectation yields
\begin{equation}
\mathbb{E}[QS(f,X)]=||f-p||_2^2-||p||_2^2.
\label{eqn:qs1}
\end{equation}
We can now write $f(x)=p(x)+\gamma(x)$, where $\int\gamma(x)\ud x=0$, and substitute it into~(\ref{eqn:qs1}) to obtain
\begin{equation*}
\mathbb{E}[QS(f,X)]=||\gamma||_2^2-||p||_2^2
\end{equation*}
As was the case with the Brier score, the functions $\pm\gamma(x)$ yield the same quadratic score. For any two forecasts, $f_i(x)=p(x)+\gamma_i(x)$, $i=1,2$ with $||\gamma_1||_2<||\gamma_2||_2$, the quadratic scoring rule would prefer $f_1(x)$. Further more, $||\gamma_1||=||\gamma_2||$ implies that $\mathbb{E}[QS(f_1,X)]=\mathbb{E}[QS(f_2,X)]$.
\subsection{The Logarithmic scoring rule}
The expectation of the logarithmic scoring rule for the forecast is
\begin{equation*}
\mathbb{E}[\ign(f,X))]=-\int p(x)\log(p(x)+\gamma(x))\ud x.
\end{equation*}
As in the discrete case, we introduce the pdfs $f_+(x)=p(x)+\gamma(x)$, $f_-(x)=p(x)-\gamma(x)$ so that we can define $\mathbb{E}[\ign]_{\pm}=\mathbb{E}[\ign(f_+,X))]-\mathbb{E}[\ign(f_-,X))]$. It follows that
\begin{equation}
\mathbb{E}[\ign]_{\pm}=\int p(x)\log\left(\frac{p(x)-\gamma(x)}{p(x)+\gamma(x)}\right)\ud x.
\label{eqn:ign4}
\end{equation}
It is necessary that $|\gamma(x)|\le p(x)$ for~(\ref{eqn:ign4}) to be well defined. Consider the case when $p(x)=p(-x)$. If, in addition, $\gamma(x)$ is an odd function, i.e. $\gamma(-x)=-\gamma(x)$, then equation~(\ref{eqn:ign4}) yields $\mathbb{E}[\ign]_{\pm}=0.$


When $\gamma(-x)=-\gamma(x)$ and $\int_{\infty}^0p(x)\ud x>0.5$, we state the following proposition:
\begin{proposition}
\label{prop5}
Given that $\gamma(-x)=-\gamma(x)$ with $\gamma(|x|)<0$ and $p(|x|)\le p(x)$, then $\mathbb{E}[\ign]_{\pm}\ge 0$. 
\end{proposition}
 The above proposition gives conditions under which the forecast $f_-(x)$ is preferred by the logarithmic scoring rule over $f_+(x)$.

\begin{proof}
The proof proceeds as follows:
\begin{eqnarray*}
\mathbb{E}[\ign]_{\pm}&=&\int_{-\infty}^{\infty} p(x)\log\left(\frac{p(x)-\gamma(x)}{p(x)+\gamma(x)}\right)\ud x\\
&=&\int_{-\infty}^0 p(x)\log\left(\frac{p(x)-\gamma(x)}{p(x)+\gamma(x)}\right)\ud x+\int_0^{\infty} p(x)\log\left(\frac{p(x)-\gamma(x)}{p(x)+\gamma(x)}\right)\ud x
\end{eqnarray*}
If we now perform a change of variable $u=-x$ in the right hand integral and then replace $u$ by $x$, we obtain
\begin{eqnarray*}
\mathbb{E}[\ign]_{\pm}&=&\int_{-\infty}^0 p(x)\log\left(\frac{p(x)-\gamma(x)}{p(x)+\gamma(x)}\right)\ud x-\int_0^{-\infty} p(-x)\log\left(\frac{p(-x)-\gamma(-x)}{p(-x)+\gamma(-x)}\right)\ud x\\
&=&\int_{-\infty}^0 p(x)\log\left(\frac{p(x)-\gamma(x)}{p(x)+\gamma(x)}\right)\ud x+\int_{-\infty}^0 p(-x)\log\left(\frac{p(-x)+\gamma(x)}{p(-x)-\gamma(x)}\right)\ud x\\
&\ge&\int_{-\infty}^0 p(x)\log\left(\frac{p(x)-\gamma(x)}{p(x)+\gamma(x)}\right)\ud x+\int_{-\infty}^0 p(x)\log\left(\frac{p(x)+\gamma(x)}{p(x)-\gamma(x)}\right)\ud x=0,
\end{eqnarray*}
where we used $p(|x|)\le p(x)$ to obtain the last inequality. To justify the use of this inequality, we need to show that the function
\begin{equation*}
\Phi(p)=p\log\left(\frac{p+\gamma}{p-\gamma}\right)
\end{equation*}
is a decreasing function for $\gamma\in(0,p)$. Differentiating $\Phi$ with respect to $p$ yields
\begin{equation*}
\Phi'(p)=\log\left(\frac{p+\gamma}{p-\gamma}\right)-\frac{2p\gamma}{p^2-\gamma^2}.
\end{equation*}
It now suffices to show that $\Phi'(p)<0$ for all $p$. Let us introduce the notation $W(p)=\log[(p+\gamma)/(p-\gamma)]$ and $Y(p)=2p\gamma/(p^2-\gamma^2)$ so that $\Phi'(p)=W(p)-Y(p)$. Note that $W(2\gamma)=\log 2$ and $Y(2\gamma)=4/3=\log(e^{4/3})$. Hence $W(2\gamma)<Y(2\gamma)$, which implies that $\Phi'(2\gamma)<0$. Differentiating $W(p)$ and $Y(p)$ with respect to $p$ yields 
\begin{equation*}
W'(p)=\frac{-2\gamma}{p^2-\gamma^2}\quad\mbox{and}\quad Y'(p)=\frac{-2\gamma(p^2+\gamma^2)}{(p^2-\gamma^2)^2}.
\end{equation*}
It is now clear that $W'(p)<0$ and $Y'(p)<0$ for all $p$. Further more, $Y'(p)<W'(p)$. Hence $W(p)<Y(p)$ for all $p\in(\gamma,2\gamma]$, which implies that $\Phi'(p)<0$ for all $p\in(\gamma,2\gamma]$. It now remains to be shown that $\Phi'(p)<0$ for all $p\in(2\gamma,\infty)$. It suffices to consider the asymptotic behaviour as $p\rightarrow\infty$. Applying L'Hopital's rule, we obtain 
\begin{equation*}
\lim_{p\rightarrow\infty}\frac{|W(p)|}{|Y(p)|}=\lim_{p\rightarrow\infty}\frac{|W'(p)|}{|Y'(p)|}=\lim_{p\rightarrow\infty}\frac{p^2+\gamma^2}{p^2-\gamma^2}=1.
\end{equation*}
Hence, $\lim_{p\rightarrow\infty}W(p)=\lim_{p\rightarrow\infty}Y(p)$, i.e. $W(\infty)=Y(\infty)$. With this result in mind, for all $p\in[2\gamma,\infty)$, we have
\begin{align*}
\int_p^{\infty}Y'(\tau)\ud\tau<\int_p^{\infty}W'(\tau)\ud\tau&\Rightarrow \left. Y(\tau)\right|_p^{\infty}<\left. W(\tau)\right|_p^{\infty}\\
&\Rightarrow Y(\infty)-Y(p)<W(\infty)-W(p)\\
&\Rightarrow W(p)<Y(p),
\end{align*}
 which completes the proof. The condition $p(|x|)<p(x)$ implies that $\int_{-\infty}^0p(x)\ud x\ge0.5$. It corresponds to the case $p>q$ in the discrete case.
\end{proof}

We now want to compare the entropies of the forecasts $f(x)=p(x)\pm\gamma(x)$ when $\gamma(-x)=-\gamma(x)$ and $\gamma(|x|)\le\gamma(x)$. The entropy of the function $f(x)=p(x)+\gamma(x)$ is then given by
\begin{equation}
h(\gamma)=-\int (p(x)+\gamma(x))\log(p(x)+\gamma(x))\ud x.
\label{eqn:ign7}
\end{equation} 
The functional derivative of $h(\gamma)$ with respect to $\gamma$ is then given by
\begin{eqnarray}
\nonumber
 \frac{\delta h(\gamma)}{\delta\gamma(x)}&=&-\frac{\partial}{\partial\gamma(x)}\left\{(p(x)+\gamma(x))\log(p(x)+\gamma(x))\right\}\\
&=&-[\log(p+\gamma)+1].
\label{eqn:ign8}
\end{eqnarray}
The order $O(\varepsilon)$ part of $h(\gamma+\varepsilon\delta\gamma)-h(\gamma)$ is given by~(see~\cite{stone} for further insights)
\begin{eqnarray}
\delta h(\gamma)&=&\int_{-\infty}^{\infty}\frac{\delta h(\gamma)}{\delta\gamma(x)}\delta\gamma(x)\ud x.
\label{eqn:ign9}
\end{eqnarray}
Plugging~(\ref{eqn:ign8}) into~(\ref{eqn:ign9}) yields
\begin{align*}
\delta h(\gamma)&=-\int_{-\infty}^{\infty}[\log(p(x)+\gamma(x))+1]\delta\gamma(x)\ud x\displaybreak[0]\\
&=-\int_{-\infty}^0[\log(p(x)+\gamma(x))+1]\delta\gamma(x)\ud x-\int_0^{\infty}[\log(p(x)+\gamma(x))+1]\delta\gamma(x)\ud x\displaybreak[0]\\
&=-\int_{-\infty}^0[\log(p(x)+\gamma(x))+1]\delta\gamma(x)\ud x+\int_0^{-\infty}[\log(p(-x)+\gamma(-x))+1]\delta\gamma(-x)\ud x\displaybreak[0]\\
&=-\int_{-\infty}^0[\log(p(x)+\gamma(x))+1]\delta\gamma(x)\ud x-\int_{-\infty}^{0}[\log(p(-x)+\gamma(-x))+1]\delta\gamma(-x)\ud x\displaybreak[0]\\
&=-\int_{-\infty}^0[\log(p(x)+\gamma(x))+1]\delta\gamma(x)\ud x+\int_{-\infty}^{0}[\log(p(-x)-\gamma(x))+1]\delta\gamma(x)\ud x\displaybreak[0]\\
&=-\int_{-\infty}^0\log\left(\frac{p(x)+\gamma(x)}{p(-x)-\gamma(x)}\right)\delta\gamma(x)\ud x,
\end{align*}
where we have applied a change of variable $x\rightarrow-x$ in the second integral of the third line and assumed $\delta\gamma(-x)=-\delta\gamma(x)$ in the fifth line. In particular,
\begin{eqnarray*}
\delta h(\gamma)|_{\gamma=0}&=&-\int_{-\infty}^0\log\left(\frac{p(x)}{p(-x)}\right)\delta\gamma(x)\ud x.
\end{eqnarray*}
Using the assumption that $p(x)\ge p(-x)$ whenever $x<0$, we consequently obtain
\begin{equation}
\delta h(\gamma)|_{\gamma=0}\le 0,
\end{equation}
if $\delta\gamma(x)>0$ for all $x<0$. In effect, we have just proved the following proposition:
\begin{proposition}
\label{prop6}
Given that $\gamma(-x)=-\gamma(x)$, $\int\gamma(x)\ud x=0$, $\gamma(|x|)\le 0$, $p(|x|)\le p(x)$ and $|\gamma(x)|<p(x)$, then the entropy of the forecast density $f_+(x)=p(x)+\gamma(x)$ is lower than that of the forecast density $f_-(x)=p(x)-\gamma(x)$.
\end{proposition}
Propositions~\ref{prop5} and~\ref{prop6} imply that the logarithmic scoring rule prefers the forecast density that is less informative, which is in agreement with the categorical case considered in the previous section.
\begin{proposition}
Given two forecasts $f_i(x)=p(x)+\gamma_i(x)$, $i=1,2$, with (i) $|\gamma_1(x)|<|\gamma_2(x)|$, (ii) $\gamma_i(|x|)\le0$, (iii) $\gamma_i(-x)=-\gamma_i(x)$, (iv) $|\gamma_i(x)|\le p(x)$ and (v) $p(|x|)\le p(x)$, then the logarithmic scoring rule prefers forecast $f_1(x)$ over forecast $f_2(x)$.
\end{proposition}
\begin{proof}
To prove the above proposition, we consider the functional derivative of the expected logarithmic scoring rule, $\mathbb{E}[\ign]=-\int_{-\infty}^{\infty}p(x)\log(p(x)+\gamma(x))\ud x$. The functional derivative with respect to $\gamma(x)$ is
\begin{equation*}
\frac{\delta}{\delta\gamma}\mathbb{E}[\ign]=\frac{-p(x)}{p(x)+\gamma(x)}.
\end{equation*}
Using this result, we obtain the first variation of $\mathbb{E}[\ign]$ as
\begin{align*}
\delta\mathbb{E}[\ign]&=\int_{-\infty}^{\infty}\frac{\delta\mathbb{E}[\ign]}{\delta\gamma(x)}\delta\gamma(x)\ud x\displaybreak[0]\\
&=\int_{-\infty}^{\infty}\frac{-p(x)}{p(x)+\gamma(x)}\delta\gamma(x)\ud x\displaybreak[0]\\
&=\int_{-\infty}^{0}\frac{-p(x)}{p(x)+\gamma(x)}\delta\gamma(x)\ud x+\int_{0}^{\infty}\frac{-p(x)}{p(x)+\gamma(x)}\delta\gamma(x)\ud x\displaybreak[0]\\
&=\int_{-\infty}^{0}\frac{-p(x)}{p(x)+\gamma(x)}\delta\gamma(x)\ud x+\int_{0}^{-\infty}\frac{p(-x)}{p(-x)+\gamma(-x)}\delta\gamma(-x)\ud x\displaybreak[0]\\
&=\int_{-\infty}^{0}\frac{-p(x)}{p(x)+\gamma(x)}\delta\gamma(x)\ud x+\int_{-\infty}^{0}\frac{p(-x)}{p(-x)-\gamma(x)}\delta\gamma(x)\ud x\displaybreak[0]\\
&=\int_{-\infty}^{0}\left[\frac{p(-x)}{p(-x)-\gamma(x)}-\frac{p(x)}{p(x)+\gamma(x)}\right]\delta\gamma(x)\ud x\displaybreak[0]\\
&=\int_{-\infty}^{0}\frac{[p(-x)+p(x)]\gamma(x)}{[p(-x)-\gamma(x)][p(x)+\gamma(x)]}\delta\gamma(x)\ud x\displaybreak[0]\\
&\ge0,
\end{align*}
provided $\delta\gamma(x)>0$ for all $x<0$, $\delta\gamma(-x)=-\delta\gamma(x)$,  $\gamma(-x)=-\gamma(x)$ and $\gamma(|x|)\le 0$. What has been shown is that as $\gamma(x)$ changes by $\delta\gamma(x)$, the expected logarithmic score changes by a positive amount. In particular, if we start at $\gamma(x)=\gamma_1(x)$, and progressively move towards $\gamma(x)=\gamma_2(x)$ by making successive additions of $\delta\gamma(x)$, the expected logarithmic score can only increase. Hence the expected logarithmic score of $\gamma_2(x)$ will be higher than that of $\gamma_2(x)$, which yields the result.
\end{proof}

We shall now consider two forecasts, $f_1(x)=p(x)+\gamma_1(x)$ and $f_2(x)=p(x)-\gamma_2(x)$ with $|\gamma_1(x)|\le|\gamma_2(x)|\le p(x)$. In this case, the quadratic scoring rule would prefer $f_1(x)$ over $f_2(x)$. In order to determine which forecast the logarithmic scoring would prefer, we consider the functional
\begin{equation}
\mathcal{H}(\gamma_1,\gamma_2)=\int_{-\infty}^{\infty}p(x)\log\left(\frac{p(x)-\gamma_2(x)}{p(x)+\gamma_1(x)}\right)\ud x.
\end{equation}
Then the following proposition holds
\begin{proposition}
Given that $|\gamma_1(x)|\le|\gamma_2(x)|\le p(x)$ and $\gamma_i(-x)=-\gamma_i(x)$, $i=1,2$, there exists $\gamma*(x)$ satisfying the inequalities $\gamma^*(x)\gamma_2(x)\ge0$ and $|\gamma^*(x)|\le|\gamma_2(x)|$ such that (a) $\mathcal{H}(\gamma^*,\gamma_2)=0$, (b) $\mathcal{H}(\gamma_1,\gamma_2)>0$ for $|\gamma^*|<|\gamma_1|$ and (c) $\mathcal{H}(\gamma_1,\gamma_2)<0$ for $|\gamma^*|>|\gamma_1|$.
\end{proposition}
\begin{proof}
It is helpful to first note that Proposition~\ref{prop5} implies that $\mathcal{H}(\gamma_2,\gamma_2)>0$ when $\gamma_2\neq0$.~Thinking of $\gamma_1(x)$ as fixed, the first variation of $\mathcal{H}(\cdot,\gamma_2)$ with respect to $\gamma_2(x)$ is given by
\begin{align*}
\delta\mathcal{H}(\cdot,\gamma_2)&=\int_{-\infty}^{\infty}\frac{\delta\mathcal{H}(\cdot,\gamma_2)}{\delta\gamma_2(x)}\delta\gamma_2(x)\ud x\\
&=\int_{-\infty}^{\infty}\frac{-p(x)}{p(x)-\gamma_2(x)}\delta\gamma_2(x)\ud x\\
&=\int_{-\infty}^{0}\frac{-p(x)}{p(x)-\gamma_2(x)}\delta\gamma_2(x)\ud x+\int_{0}^{\infty}\frac{-p(x)}{p(x)-\gamma_2(x)}\delta\gamma_2(x)\ud x\displaybreak[0]\\
&=\int_{-\infty}^{0}\frac{-p(x)}{p(x)-\gamma_2(x)}\delta\gamma_2(x)\ud x-\int_{0}^{-\infty}\frac{p(-x)}{p(-x)+\gamma_2(x)}\delta\gamma_2(x)\ud x\displaybreak[0]\\
&=\int_{-\infty}^{0}\frac{-p(x)}{p(x)-\gamma_2(x)}\delta\gamma_2(x)\ud x+\int_{-\infty}^{0}\frac{p(-x)}{p(-x)+\gamma_2(x)}\delta\gamma_2(x)\ud x\displaybreak[0]\\
&=\int_{-\infty}^{0}\left[\frac{p(-x)}{p(-x)+\gamma_2(x)}-\frac{p(x)}{p(x)-\gamma_2(x)}\right]\delta\gamma_2(x)\ud x\displaybreak[0]\\
&=\int_{-\infty}^{0}\frac{-[p(-x)+p(x)]\gamma_2(x)}{[p(-x)+\gamma_2(x)][p(x)-\gamma_2(x)]}\delta\gamma_2(x)\ud x\displaybreak[0]\\
&\le0,
\end{align*}
provided $\delta\gamma_2>0$ and $\delta\gamma_2(-x)=-\delta\gamma_2(x)$. In the fourth line a change of variable $x=-\tau$ was applied and then $\tau$ was replaced with $x$ since it is a dummy variable. It follows that $\mathcal{H}(\cdot,\gamma_2)$ has a maximum when $\gamma_2=0$, i.e.  $\mathcal{H}(\cdot,\gamma_2)\le\mathcal{H}(\cdot,0)$. In particular, $\mathcal{H}(0,\gamma_2)\le\mathcal{H}(0,0)=0$. For $\gamma_2\neq0$, we have the strict inequality,  $\mathcal{H}(0,\gamma_2)<0$. Since $\mathcal{H}(\gamma_2,\gamma_2)>0$, continuity implies that $\mathcal{H}(\gamma_1,\gamma_2)=0$ for some $\gamma_1(x)=\gamma^*(x)$ such that $|\gamma^*|<|\gamma_2|$, and this completes the proof. 
\end{proof}
\subsection{The Spherical Scoring Rule}
Given a forecast $f(x)$, the spherical scoring rule is given by
\begin{equation*}
S(f,X)=-\frac{f(X)}{||f||_2}.
\end{equation*}
If we define the operator $\rho f=f(x)/||f||_2$, the expected spherical score is the inner product
\begin{equation*}
\mathbb{E}[S(f,X)]=-\langle\rho f,p\rangle.
\end{equation*}
The minimum of this expectation is achieved if and only $f=p$ since it is a strictly proper scoring rule~\citep{frie-83}. We now state the following proposition:
\begin{proposition}
Given that $\gamma(-x)=-\gamma(x)$ with $\gamma(|x|)<0$, $|\gamma(x)|<p(x)$ and $p(|x|)\le p(x)$, then the spherical scoring rule prefers the forecast $f_+(x)$ over $f_-(x)$, i.e. $\mathbb{E}[S(f_+,X)]\le\mathbb{E}[S(f_-,X)]$ .
\end{proposition}
\begin{proof}
The aim here is to show that $\mathbb{E}[S(f_+,X)]\le\mathbb{E}[S(f_-,X)]$, which is equivalent to $\langle\rho f_+,p\rangle\ge\langle\rho f_-,p\rangle$. Note that each of these inner products is non-negative since
\begin{eqnarray*}
\langle\rho f_\pm,p\rangle&=&\frac{\langle f_\pm,p\rangle}{||f_\pm||_2}\\
&=&\frac{\langle p\pm\gamma,p\rangle}{||f_\pm||_2}\\
&=&\frac{||p||_2^2\pm\langle\gamma,p\rangle}{||f_\pm||_2}\\
&\ge&0,
\end{eqnarray*}
due to Cauchy Schwartz's inequality, $\langle\pm\gamma,p\rangle\le||\gamma||_2||p||_2$, and the hypothesis, $|\gamma(x)|\le p(x)\Rightarrow||\gamma||_2\le||p||_2$. Therefore,  $\langle\rho f_+,p\rangle\ge\langle\rho f_-,p\rangle$ is equivalent to  $\langle\rho f_+,p\rangle^2\ge\langle\rho f_-,p\rangle^2$. It therefore suffices to show that the latter inequality holds.
\begin{eqnarray*}
\langle\rho f_+,p\rangle^2-\langle\rho f_-,p\rangle^2&=&\frac{\langle f_+,p\rangle^2}{||f_+||_2^2}-\frac{\langle f_-,p\rangle^2}{||f_-||_2^2}\\
&=&\frac{\langle p+\gamma,p\rangle^2}{||f_+||_2^2}-\frac{\langle p-\gamma,p\rangle^2}{||f_-||_2^2}\\
&=&\frac{(||p||_2^2+\langle \gamma,p\rangle)^2}{||f_+||_2^2}-\frac{(||p||_2^2-\langle \gamma,p\rangle)^2}{||f_-||_2^2}\\
&=&\frac{||f_-||_2^2(||p||_2^2+\langle \gamma,p\rangle)^2-||f_+||_2^2(||p||_2^2-\langle \gamma,p\rangle)^2}{||f_+||_2^2||f_-||_2^2}.
\end{eqnarray*}
Plugging in $||f_+||_2^2=||p||_2^2+2\langle\gamma,p\rangle+||\gamma||_2^2$ and $||f_-||_2^2=||p||_2^2-2\langle\gamma,p\rangle+||\gamma||_2^2$ into the numerator of the last expression, removing brackets and collecting like terms yield
\begin{eqnarray}
\langle\rho f_+,p\rangle^2-\langle\rho f_-,p\rangle^2=\frac{\langle\gamma,p\rangle(||p||_2^2||\gamma||_2^2-\langle\gamma,p\rangle^2)}{||f_+||_2^2||f_-||_2^2}.
\label{eqn:sph}
\end{eqnarray}
As a consequence of Cauchy-Schwartz's inequality, $||p||_2^2||\gamma||_2^2-\langle\gamma,p\rangle^2\ge0$. It will now be shown that under the hypothesis of the proposition, $\langle\gamma,p\rangle\ge 0$.
\begin{align*}
\langle\gamma,p\rangle&=\int_{-\infty}^{\infty}\gamma(x)p(x)\ud x\\
&=\int_{-\infty}^0\gamma(x)p(x)\ud x+\int_0^{\infty}\gamma(x)p(x)\ud x\displaybreak[0]\\
&=\int_{-\infty}^0\gamma(x)p(x)\ud x-\int_0^{-\infty}\gamma(-x)p(-x)\ud x\displaybreak[0]\\
&=\int_{-\infty}^0\gamma(x)p(x)\ud x+\int_{-\infty}^{0}\gamma(-x)p(-x)\ud x\displaybreak[0]\\
&=\int_{-\infty}^0\gamma(x)p(x)\ud x-\int_{-\infty}^{0}\gamma(x)p(-x)\ud x\displaybreak[0]\\
&=\int_{-\infty}^0\gamma(x)[p(x)-p(-x)]\ud x\displaybreak[0]\\
&\ge0,
\end{align*}
since $p(x)\ge p(-x)$ and $\gamma(x)\ge0$ for all $x\le0$. Hence, the right hand side of equation~(\ref{eqn:sph}) is non-negative.
\end{proof}

The distribution preferred by the spherical scoring rule is already known through Proposition~\ref{prop6} to be of lower entropy. As was the case in the binary case, the spherical scoring rule prefers an opposite distribution to the logarithmic scoring rule.
\subsection{Continuous Ranked Probability Score}
Finally, we consider the {\em Continuous Ranked Probability Score} (CRPS) of the density forecast $f(x)$ whose cumulative distribution is $F(x)$. The CRPS is a function of $F$ and the verification $X$ and is defined by~\citep{til-07}
\begin{equation*}
\mathrm{CRPS}(F,X)=\int_{-\infty}^{\infty}(F(\tau)-\mathbb{I}\{\tau\ge X\})^2\ud\tau.
\end{equation*}
The above score may equivalently be written as
\begin{equation}
\mathrm{CRPS}(F,X)=\int_{-\infty}^XF^2(\tau)\ud\tau+\int_X^{\infty}(F(\tau)-1)^2\ud\tau.
\label{eqn:crps2}
\end{equation}
It follows from~(\ref{eqn:crps2}) that
\begin{equation}
\mathbb{E}[\mathrm{CRPS}(F,X)]=\int_{-\infty}^{\infty}p(x)\int_{-\infty}^xF^2(\tau)\ud\tau\ud x+\int_{-\infty}^{\infty}p(x)\int_x^{\infty}(F(\tau)-1)^2\ud\tau\ud x,
\label{eqn:crps3}
\end{equation}
where $p(x)$ is the true (or target) density function. If $P(x)=\int_{-\infty}^xp(\tau)\ud\tau$, we can then apply the integration by parts formula to each term on the right hand side of~(\ref{eqn:crps3}) to obtain
\begin{align*}
\int_{-\infty}^{\infty}p(x)\int_{-\infty}^xF^2(\tau)\ud\tau\ud x&=\left.P(x)\int_{-\infty}^xF^2(\tau)\ud\tau\right|_{-\infty}^{\infty}-\int_{-\infty}^{\infty}P(x)F^2(x)\ud x\\
&=\int_{-\infty}^{\infty}F^2(x)\ud x-\int_{-\infty}^{\infty}P(x)F^2(x)\ud x
\end{align*}
and 
\begin{align*}
\int_{-\infty}^{\infty}p(x)\int_x^{\infty}(F(\tau)-1)\ud\tau\ud x&=\left.P(x)\int_x^{\infty}(F(\tau)-1)^2\ud\tau\right|_{-\infty}^{\infty}+\int_{-\infty}^{\infty}P(x)(F(x)-1)^2\ud x\\
&=0+\int_{-\infty}^{\infty}P(x)(F(x)-1)^2\ud x,
\end{align*}
whence
\begin{equation}
\mathbb{E}[\mathrm{CRPS}(F,X)]=\int_{-\infty}^{\infty}P(x)(1-P(x))\ud x+\int_{-\infty}^{\infty}(F(x)-P(x))^2\ud x,
\label{eqn:crps4}
\end{equation}
after some algebraic manipulation. Define $F(x)=P(x)+\Gamma(x)$, where $\Gamma(x)=\int_{-\infty}^x\gamma(\tau)\ud\tau$ and $f(x)=p(x)+\gamma(x)$. If we also define $G(P)=\int P(x)(1-P(x))\ud x$, then equation~(\ref{eqn:crps4}) can be re-written as $\mathbb{E}[\mathrm{CRPS}(F,X)]=G(P)+||\Gamma||_2^2$. We have thus proved the following proposition:
\begin{proposition} 
\label{prop9}
The Continuous Ranked Probability Score does not distinguish between distributions whose cumulative errors from the target distribution are equal in the sense of the $L^2$ norm. 
\end{proposition}
Consequently, the CRPS does not distinguish between density forecasts whose errors from the target density differ by the sign. To see this, consider two forecasts which whose errors from the target density are $\gamma_1(x)$ and $\gamma_2(x)$ respectively, with $\gamma_1(x)=-\gamma_2(x)$. It then follows that
\begin{align*}
||\Gamma_1||_2^2-||\Gamma_2||_2^2&=\int_{-\infty}^{\infty}\Gamma_1^2(x)\ud x-\int_{-\infty}^{\infty}\Gamma_2^2(x)\ud x\\
&=\int_{-\infty}^{\infty}\left(\Gamma_1^2(x)-\Gamma_2^2(x)\right)\ud x\\
&=\int_{-\infty}^{\infty}\left(\Gamma_1(x)+\Gamma_2(x)\right)\left(\Gamma_1(x)-\Gamma_2(x)\right)\ud x\\
&=0,
\end{align*}
since $\gamma_1(x)=-\gamma_2(x)\Rightarrow \Gamma_1(x)+\Gamma_2(x)=0$. 

As a final remark, we note that the second term in the expectation of CRPS in~(\ref{eqn:crps4}) somewhat resembles the {\em mean squared error criterion} discussed in~\cite{gra-06}. The mean squared error of the forecast $F(x)$ is $\mathbb{E}[\Gamma^2(X)]=\int p(x)\Gamma^2(x)\ud x$. Likewise, the mean squared error criterion does not distinguish between forecasts whose errors from the target density differ by a sign (i.e. $\gamma_1(x)=-\gamma_2(x)$) because
\begin{align*}
\mathbb{E}[\Gamma_1^2(X)]-\mathbb{E}[\Gamma_2^2(X)]&=\int_{-\infty}^{\infty}p(x)\left(\Gamma_1^2(x)-\Gamma_2^2(x)\right)\ud x\\
&=\int_{-\infty}^{\infty}p(x)\left(\Gamma_1(x)+\Gamma_2(x)\right)\left(\Gamma_1(x)-\Gamma_2(x)\right)\ud x\\
&=0.
\end{align*}
\section{Discussion and Conclusions}
\label{sec:disc}
This manuscript contrasted how certain scoring rules would rank competing forecasts of specified departures from the target distribution. In the categorical case, we considered the Brier Score, the logarithmic scoring rule and the spherical scoring rule, focusing on the binary case. Given two forecasts whose errors from the target distribution differ only by the sign, we found that the logarithmic scoring rule prefers the higher entropy distribution whilst the spherical scoring rule prefers the lower entropy distribution. The Brier score does not distinguish the two distributions. The logarithmic scoring rule selects a lower entropy forecast only if it is nearer to the target distribution in the sense of the $L^2$ norm and vice versa for the spherical scoring rule. 

We extended the investigation from binary forecasts to the continuous case, where we considered the Quadratic score, Logarithmic scoring rule, spherical scoring rule and the Continuous Ranked Probability Score (CRPS). Just like the Brier score in the binary case, the Quadratic Score does not distinguish between forecasts with equal $L^2$ norms of their errors from the target distribution. On the other hand, given two density forecasts whose errors from the target forecast differ by a sign, the logarithmic scoring rule prefers the distribution with higher entropy whilst the spherical scoring rule prefers the one with lower entropy: bear in mind that higher entropy corresponds to more uncertainty~\citep{sha}. The CRPS is indifferent to forecasts whose errors from the target density differ by a sign.

Some have criticised the logarithmic scoring rule for placing a heavy penalty on assigning zero probability to events that materialise~\citep[e.g.][]{boe-11,til-07}; but assigning zero probability to events that are possible is also discouraged by {\em Laplace's rule of succession}~\citep{jay-03}. What has been shown here is that the logarithmic scoring rule is good at highlighting forecasts that are less uncertain than ideal forecasts. Such forecasts may have to be dealt with appropriately. One way of dealing with such forecasts is discussed in~\cite{mac-earl}. Nonetheless, given two density forecasts, the logarithmic scoring rule does not just reject the more extreme in the sense of entropy: If both forecasts are more uncertain that the ideal forecast, the logarithmic scoring rule will tend to prefer the less uncertain of the two.

Does our consideration of departures from ideal forecasts amount to advocating for dishonesty by forecasters? Not at all. We are merely making an observation that forecasters can honestly report predictive distributions that have departures from ideal forecasts. Although strictly proper scoring rules encourage forecasters to be honest when they report their best judgements, they do not guarantee that the reported forecasts will coincide with ideal forecasts. Our point then is that using a given scoring rule may inherently favour departures from ideal forecasts in one direction more than in another. Therefore, when one selects a scoring rule to estimate distribution parameters or choose between two competing experts, it amounts to deciding preferred departures.

Which scoring rule one should choose will depend on the application at hand. Combining insights of scoring rules set forth in this paper with an understanding of the situation at hand can help decide which scoring rule is most appropriate. The issue to consider may be decisions associated with high impact, low probability events. To illustrate our point, let us consider inflation forecasting. It is undesirable to over estimate the probability for extreme inflation because of the panic it can create as buyers rush to spend now before prices rise. In order to manage peoples expectations better, the spherical scoring rule is preferable in this case. As another example, consider seasonal forecasts of drought in the UK, which is arguably a rare event. Under estimating the probability of this event could result in water shortages since water companies might not be stringent on water usage. In this case, the logarithmic scoring rule is preferable. 
\section*{Acknowledgements}
This work was supported by the RCUK Digital Economy Programme at the University of Reading via EPSRC grant EP/G065802/1 The Horizon Digital Economy Hub.
\bibliographystyle{jofstyle}
\bibliography{refs}
\end{document}